\documentclass[a4paper,10pt]{scrartcl}

\usepackage[fleqn]{amsmath}
\usepackage{amssymb}
\usepackage{amsthm}
\usepackage{stmaryrd}
\usepackage{bbm}
\usepackage[colorlinks]{hyperref}
\usepackage{mathtools}
\usepackage[numbers]{natbib}

\mathtoolsset{showonlyrefs}

\newcommand{\bP}{\ensuremath{\mathbb{P}}}

\newcommand{\bR}{\ensuremath{\mathbb{R}}}

\newcommand{\bZ}{\ensuremath{\mathbb{Z}}}

\newcommand{\cP}{\ensuremath{\mathcal{P}}}

\newcommand{\cT}{\ensuremath{\mathcal{T}}}

\newcommand{\norm}[1]{\left\Vert \, #1 \, \right\Vert}

\newcommand{\ddx}[1][1]{\ifnum#1=1 \frac{d}{dx} \else \frac{d^{#1}}{dx^{#1}} \fi}
\newcommand{\ddy}[1][1]{\ifnum#1=1 \frac{d}{dy} \else \frac{d^{#1}}{dy^{#1}} \fi}
\newcommand{\ddt}[1][1]{\ifnum#1=1 \frac{d}{dt} \else \frac{d^{#1}}{dt^{#1}} \fi}

\theoremstyle{plain}
\newtheorem{thm}{Theorem}[section]  
\newtheorem{prop}[thm]{Proposition}
\newtheorem{cor}[thm]{Corollary}
\newtheorem{lem}[thm]{Lemma}
\newtheorem{defn}{Definition}[section]

\theoremstyle{definition}

\newtheorem{q}{Question}

\theoremstyle{remark}
\newtheorem{rem}{Remark}[section]

\numberwithin{equation}{section}

\DeclareMathOperator{\integers}{\mathbb{Z}}
\DeclareMathOperator{\reals}{\mathbb{R}}

\DeclareMathOperator{\nat}{\mathbb{N}}

\newcommand{\set}[1]{\{ #1 \}}

\newcommand{\Latd}{\mathbb{Z}^d}

\newcommand{\hP}{\widehat{\bP}}

\title{Stochastic domination in space-time for the contact process}

\author{Jacob van den Berg  \footnote{CWI Amsterdam and VU University Amsterdam. 
 Email: j.van.den.Berg@cwi.nl }  \\
 Stein Andreas Bethuelsen 
\footnote{Technische Universit\"at M\"unchen. 
Email: stein.bethuelsen@tum.de}}

\begin{document}
\maketitle

\abstract{
 Liggett and Steif (2006) proved that, for the supercritical contact process on certain graphs, the upper invariant measure stochastically dominates an i.i.d.\ Bernoulli product measure. In particular, they proved this for $\Latd$ and (for infection rate sufficiently large) $d$-ary homogeneous trees $T_d$.

In this paper we prove some space-time versions of their results. We do this by combining their methods with specific properties of the contact process and general correlation inequalities.

One of our main results concerns the contact process on $T_d$ with $d\geq2$. We show that, for large infection rate, there exists a subset $\Delta$ of the vertices of $T_d$, containing a ``positive fraction" of all the vertices of $T_d$, such that the following holds: The contact process on $T_d$ observed on $\Delta$ stochastically dominates an independent spin-flip process. (This is known to be false for the contact process on graphs having subexponential growth.)

We further prove that the supercritical contact process on $\Latd$ observed on certain $d$-dimensional space-time slabs stochastically dominates an i.i.d.\ Bernoulli product measure, from which we conclude strong mixing properties  important in the study of certain random walks in random environment.
}
\bigskip

 \bigskip

\section{Introduction and main results}

\subsection{Background and outline of this paper}
Let $G=(V,E)$ be a connected graph of bounded degree, and let $\lambda \in (0,\infty)$.  The contact process $(\eta_t)$ on $G$ with parameter $\lambda$  is the continuous-time  interacting particle system  on $\set{0,1}^{V}$ with local transition rates given by
\[ \eta \rightarrow \eta_x \text{ at rate } \left\{
	\begin{array}{ll}
		1, & \text{if }\eta(x) =1; \\
		\lambda \sum_{\{y \colon \{x,y\} \in E\}} \eta(y), & \text{if } \eta(x)=0,
	\end{array}
\right.\] for $x\in V$, and where $\eta_x$ is defined by $\eta_x(y) := \eta(y)$ for $y \neq x$, and $\eta_x(x) := 1 - \eta(x)$.
\medskip

Equivalently (see also Section \ref{sec contact}), one can imagine that for each site $x$ and each neighbour $y$ 
of $x$, there are 'clocks', denoted by $I(y,x)$ and $H(x)$,
which ring after independent, exponentially distributed, with mean $1/ \lambda$ and $1$ respectively, times (independent of the other clocks). At
each ring of the clock $I(y,x)$ the following happens: if $y$ has value $1$ 
and $x$ has value $0$, the value of $x$ immediately changes to
$1$. At each ring of the clock $H(x)$ the following happens: if $x$ has value $1$, the value of $x$ 
immediately changes to $0$.\medskip

The contact process was introduced by \citet{HarrisCP1974} in 1974 as a toy model for the spread of an infection in a population. With this interpretation in mind, $\lambda$ is often referred to as the ``infection" parameter and a site $x \in V$ is said to be \emph{infected} at time $t$ if $\eta_t(x)=1$, otherwise it is  said to be \emph{healthy}. A central question is whether the infections ``survive" with positive probability or eventually die out, i.e.\ all sites become healthy. As general references on contact processes we mention the books \cite{LiggettIPS1985} and \cite{LiggettSIS1999} by Liggett, from which we next recall some well
known properties. \medskip

The ``healthy" configuration where all sites are  equal to $0$, denoted by $\bar{0}$, 
is clearly  an absorbing state for the contact process. On the other hand, starting from  the full configuration where all sites are initially infected,
the contact process  evolves towards an invariant measure $\bar{\nu}_{\lambda}$. This state is often called the \emph{upper invariant measure}. 
Throughout this text we write \emph{upper stationary contact process} to denote the contact process whose law at an arbitrary time equals $\bar{\nu}_{\lambda}$.
\medskip

A well-known property of the contact process is that, if $G$ is countable infinite, it undergoes a phase transition: there is a critical threshold $\lambda_c \in [0,\infty)$, depending on $G$, such that, for all $\lambda <\lambda_c$, $\bar{\nu}_{\lambda}= \delta_{\bar{0}}$, and for all $\lambda>\lambda_c$, $\bar{\nu}_{\lambda}\neq \delta_{\bar{0}}$. Here, $\delta_{\bar{0}}$ denotes the measure that concentrates on $\bar{0}$.
In this paper we focus on the supercritical phase, i.e., the case $\lambda>\lambda_c$.
\medskip

Van den Berg, H\"aggstr\"om and Kahn \cite{BergHaggstromKahn2006} proved that the upper invariant measure satisfies the following property (called downward FKG in \cite{LiggettSteifSD2006}): for any finite $\Delta \subset V$,  the conditional measure $\bar{\nu}_{\lambda}(\cdot \mid \eta \equiv 0 \text{ on } \Delta)$ is positively associated. (See \citet{LiggettCASS2006} for a slightly stronger property).
\citet{LiggettSteifSD2006} used this result to show that, for the supercritical contact process  on $\Latd$, $d\geq1$, the upper invariant measure stochastically dominates a non-trivial Bernoulli product measure (see Corollary 4.1 therein). For the $d$-ary homogeneous tree $T_d$, where each site has $d+1$ neighbouring sites, they showed such a domination result for $\lambda>4$.
\medskip

 In this paper, we investigate if analogs of the mentioned domination results by Liggett
     and Steif hold for the contact process observed in space \emph{and time}.
  \medskip

 One of our main results concerns the upper stationary contact process on $T_d$ with $d\geq2$. We show that, for $\lambda > \lambda_c(\integers)$, there exists a subset $V$ of the vertices of $T_d$, containing a ``positive fraction" of all the vertices of $T_d$, such that the following holds: the contact process on $T_d$ observed on $V$ stochastically dominates a non-trivial independent spin-flip process (see Section \ref{sec sd and Bpm} for a definition of such processes). This is the content of Theorem \ref{thm spin-flip tree} below. Interestingly, this cannot happen for the upper stationary contact process on graphs having subexponential growth (such as $\bZ^d$), as shown in Proposition \ref{prop amenable}. \medskip

We furthermore prove that the upper stationary contact process on $\Latd$ with $d\geq1$ and $\lambda>\lambda_c$, observed on certain (discrete-time) $d$-dimensional space-time slabs, stochastically dominates a non-trivial Bernoulli product measure. This is the content of  Theorem \ref{thm spin-flip 2} below. Using this, we conclude in Theorem \ref{prop cone-mixing} that the contact process projected onto a thin space-time slab satisfies a strong mixing property known as cone-mixing. 
\medskip

The projection of the contact process onto a sub-lattice can be interpreted as a hidden Markov model and is motivated for instance by the study of  phase transition phenomena in nonlinear filtering (see \citet{RebeschiniHandelPTNF2015}) as well as the study of a random walk in a dynamic random environment (see \citet{BethuelsenVolleringRWDRE2016}). \medskip

A key observation for our arguments is that the results in \cite{BergHaggstromKahn2006} imply that the above mentioned downward FKG property extends to the contact process observed in space-time (see Lemma \ref{lem DFKG for CP}).
Our proofs are based on this observation, together with specific properties of the contact process and results and techniques from \cite{LiggettSteifSD2006}.

\subsubsection*{Outline of this paper}
In the next subsection we recall some basic definitions before we present our main results for the contact process in Subsection \ref{sec results}. In Subsection \ref{sec mixing} we discuss certain mixing properties which follow from our main results. 
 Section \ref{sec preliminaries} is devoted to some preliminary results. 
Proofs of our main results are provided in Section \ref{sec proofs}.
In Section \ref{sec questions} we present some open questions.

\subsection{Stochastic domination and Bernoulli product measures}\label{sec sd and Bpm}

Besides the contact process there are two key concepts in the presentation of our main results, namely stochastic domination and Bernoulli product measures. For the convenience of the reader, we briefly recall their definitions.\medskip

Given a countable set $V$, we are interested in probability measures on $\Omega := \{0,1\}^V$ and $D_{\Omega}[0,\infty)$, the set of c\`adl\`ag functions on $[0,\infty)$ taking values in $\Omega$.
For this, denote by $\mathcal{F}$ the product $\sigma$-algebra corresponding to $\Omega$ and let $\mathcal{M}_1(\Omega)$ be the set of probability measures on $(\Omega,\mathcal{F})$, and similarly, let $ \mathcal{M}_1(D_{\Omega}[0,\infty))$ be the set of measures on $D_{\Omega}[0,\infty)$. \medskip

For $\rho\in [0,1]$, we denote by $\mu_{\rho} \in \mathcal{M}_1(\Omega)$ the \emph{Bernoulli product measure} with density $\rho$. That is, for any finite $\Delta,\Lambda \subset V$ such that $\Delta \cap \Lambda = \emptyset$, the measure $\mu_{\rho}$ has cylinder probabilities given by
\begin{align}
\mu_{\rho} \left(\eta \in \Omega \colon \eta(x)=1 \: \forall x\in \Delta, \eta(x)=0 \: \forall x \in \Lambda \right) = \rho^{|\Delta|}(1-\rho)^{|\Lambda|}.
\end{align}
A related object is the following continuous-time process. Given $\alpha\geq0$,  the \emph{independent spin-flip process} $(\xi_t)$  with parameter $\alpha$ is the continuous-time Markov process on $\{0,1\}^{V}$  with local transition rates given by
\[ \eta \rightarrow \eta_x \text{ at rate } \left\{
	\begin{array}{ll}
		1, & \text{if }\eta(x) =1; \\
		\alpha, & \text{if } \eta(x)=0,
	\end{array}
\right.\]
for $x\in V$. Note that $(\xi_t)$ is ergodic with unique invariant measure $\mu_{\rho}$, where $\rho = \rho(\alpha) = \alpha/(\alpha+1)$.
\medskip

We next introduce the concept of stochastic domination.
For this, we associate to $\Omega$ the partial ordering such that $\xi \leq \eta$  if and only if $\xi(x) \leq \eta(x)$ for all $x \in V$. An event $B \in \mathcal{F}$ is said to be \emph{increasing} if $\xi\leq\eta$ implies $1_{B}(\xi)\leq 1_{B}(\eta)$. If $\xi\leq\eta$ implies $1_{B}(\xi)\geq 1_{B}(\eta)$ then $B$ is called \emph{decreasing}. For $\mu_1, \mu_2 \in \mathcal{M}_1(\Omega)$ we say that $\mu_1$ \emph{stochastically dominates} $\mu_2$ if $\mu_2(B) \leq \mu_1(B)$ for all increasing events $B \in \mathcal{F}$. Recall that, by Strassen's theorem (see \cite{LiggettIPS1985}, p.\ 72), $\mu_1$ \emph{stochastically dominates} $\mu_2$ is  equivalent to the existence of a coupling $(\eta,\xi)$ so that $\eta$ has distribution $\mu_2$ and $\xi$ has distribution $\mu_1$, and $\eta \leq \xi$ a.s. The definition of stochastic domination readily translates to measures on $D_{\Omega}[0,\infty)$, by extending the partial ordering for elements in $\Omega$ to $\Omega^{[0,\infty)}$, requiring that $\xi_t(x) \leq \eta_t(x)$  for all $(x,t) \in V \times [0,\infty)$.
\medskip

Another key concept used in the proof of the following theorems is that of \emph{downward FKG}, to which we return to in Section \ref{sec preliminaries}.

\subsection{Main results}\label{sec results}

As shown in \citet{LiggettSteifSD2006}, Corollary 4.1, the upper stationary contact process on $\Latd$, $d \geq 1$, with $\lambda>\lambda_c$ stochastically dominates a non-trivial Bernoulli product measure when observed at a fixed time $t$. That is, $\bar{\nu}_{\lambda}$ stochastically dominates $\mu_{\rho}$ for some $\rho\in(0,1)$. \medskip

On the other hand, as also shown in \cite{LiggettSteifSD2006}, stochastic domination of a non-trivial Bernoulli product measure does not hold in general for the entire space-time evolution. This can be extended to the contact process on graphs having subexponential growth.\medskip

Let $G=(V,E)$ be a connected graph of bounded degree and denote by $d \colon V\times V \rightarrow \integers_{\geq 0}$ the graph distance on G. 
Following  \cite[p.\ 181]{LyonsPeresTrees2016}, the graph $G$ is said to have \emph{subexponential growth} (of balls) if
\begin{align}\label{eq subexponential growth}
\liminf_{n \rightarrow \infty} \left| \{ x \in V \colon d(o,x) \leq n \} \right|^{1/n} =1, \quad \text{ for some }o \in V, 
\end{align}
 where $|\cdot|$ denotes the cardinality. Otherwise $G$ is said to have \emph{exponential growth}. 
Further, we say that $\Delta \subset V$ has \emph{positive density} if 
\begin{align}\label{eq density of set} 
 \liminf_{n \rightarrow \infty} \frac{ \left|\{ x \in \Delta \colon d(o,x) \leq n \}\right|}{\left|\{ y \in V \colon d(o,y) \leq n\} \right|}> 0, \quad \text{ for some }o \in V.\end{align}
\begin{rem}
Since we assume that $G$ is connected and has bounded degree we may in \eqref{eq subexponential growth} and \eqref{eq density of set} replace ``for some $o \in V$'' by ``for all $o \in V$''.
\end{rem}

\begin{prop}\label{prop amenable}
Let $(\eta_t)$ be the upper stationary contact process on a connected graph $G=(V,E)$ having subexponential growth and bounded degree with $\lambda>0$. Consider $\Delta \subset V$  having positive density. Then, for \textbf{no} parameter value except $\alpha=0$ can $(\eta_t)$ and $(\xi_t)$ be coupled so that, when initialised from $\bar{\nu}_{\lambda}$ and $\mu_{\rho(\alpha)}$ respectively, it holds that
\begin{align} \hP \left(\eta_t(x) \geq \xi_t(x) \text{ for all } (x,t) \in  \Delta \times[0,\infty)\right)=1.\end{align}
\end{prop}

The proof of Proposition \ref{prop amenable} follows by an almost direct extension of the proof of \cite[Proposition 1.1]{LiggettSteifSD2006}, and is given in Section \ref{sec proofs 1}. In fact, the proof also works if Condition \eqref{eq subexponential growth} is replaced by the following condition,
\begin{align}\label{eq extend subexp}
\liminf_{n\rightarrow \infty} \frac{|\{ (x,y) \in E \colon d(o,x)=n =d(o,y)-1\}| }{ |\{ x \in V \colon d(o,x)\leq n\}|}=0,\: \text{ for some }o\in V,
\end{align}
which is easily seen to be weaker than \eqref{eq subexponential growth}. 
A natural question is whether Proposition \ref{prop amenable} also holds if \eqref{eq extend subexp} does not hold. 
Theorem \ref{thm spin-flip tree} below states that this is not the case  for homogeneous trees. See also Question $2$ in Section \ref{sec questions}.
\medskip

\begin{thm}\label{thm spin-flip tree}
Let $(\eta_t)$ be the upper stationary contact process on $T_d$, $d\geq2$, with $\lambda> \lambda_c(\integers)$. Let $V$ be the set of vertices of $T_d$.
Then there is a $\Delta \subset V$  having positive density  
together with an $\alpha= \alpha(\lambda)>0$  and a coupling $\hP$ of $(\eta_t)$ and $(\xi_t)$, 
  initialised from $\bar{\nu}_{\lambda}$ and $\mu_{\rho(\alpha)}$ respectively, such that
\begin{align}\label{eq projected onto tree}\hP \left(\eta_t(x) \geq \xi_t(x) \text{ for all } (x,t) \in \Delta \times[0,\infty) \right)=1.\end{align}
\end{thm}
Thus, \eqref{eq projected onto tree} in Theorem \ref{thm spin-flip tree} concerns  the contact process on $T_d$ \emph{projected onto} a subset of $V \times [0,\infty)$ (a terminology we often refer to later). Theorem \ref{thm spin-flip tree} says that the contact process projected onto $\Delta \times [0,\infty)$ stochastically dominates an independent spin-flip process.
\medskip

To prove Theorem \ref{thm spin-flip tree}, we first show that the contact process on $\{0,1,\dots \}$ observed at the vertex $0$ stochastically dominates an independent spin-flip process (in fact, we show a generalisation of this). Once this is obtained, Theorem \ref{thm spin-flip tree} follows by a monotonicity argument. From the precise argument, given in Section \ref{sec proof of spt}, it moreover follows that the set $\Delta$ in Theorem \ref{thm spin-flip tree} can be chosen such that the l.h.s.\ of \eqref{eq density of set} equals $\frac{d-1}d$.
\medskip

Denote by \begin{align}\label{eq results survival}
\tau^x:= \inf \{ t\geq 0 \colon \eta_t^x \equiv \bar{0} \}, \quad x \in V,
\end{align} 
the extinction time for the contact process $(\eta_t^x)$ started with only $x$ initially infected.

\begin{thm}\label{thm spin flip 3}
Let $(\eta_t)$ be the upper stationary contact process on a connected graph $G=(V,E)$ having bounded degree  with $\lambda>0$. Let $x \in V$ for which there exist $C,c>0$ such that, 
\begin{align}\label{eq survival2}
 &\bP ( \tau^x=\infty) >0;
\\&\label{eq survival}  \bP ( s<\tau^x<\infty) \leq Ce^{-cs}, \quad  \text{ for all } s\geq0.
\end{align}
Then there exist $\alpha= \alpha(\lambda)>0$ and a coupling $\hP$ of $(\eta_t)$ and $(\xi_t)$  initialised from $\bar{\nu}_{\lambda}$ and $\mu_{\rho(\alpha)}$ respectively,  such that 
\begin{align}\hP \left(\eta_t(x) \geq \xi_t(x) \text{ for all } t\in [0,\infty) \right)=1.\end{align}
\end{thm}
\medskip

Note that \eqref{eq survival2} and \eqref{eq survival} are known to hold for all vertices throughout the supercritical phase  for the contact process on $\Latd$, $d\geq1$ (see \cite[Theorem 1.2.30]{LiggettSIS1999}),  and on $\{0,1,\dots\}$ (see \cite{DurrettGriffeathCPhighDim1982}, p.\ 546 and \cite{DurretGriffeathCP1983}).\medskip

For the proof of Theorem \ref{thm spin flip 3} (in Section \ref{sec proofs 3}) we use that the contact process satisfies the downward FKG property in space-time (see Section \ref{sec preliminaries} for a proper definition). Combining this with large deviation estimates of the probability that there are no infections at the site $x$ in the time interval $[0,t]$ and a general theorem in \cite{LiggettSteifSD2006} (which we state in Lemma \ref{lem DFKG domi1}) yields the statement of Theorem \ref{thm spin flip 3}.   \medskip

It seems natural that Theorem \ref{thm spin flip 3} can be extended to the case where instead of observing the contact process at a single site, we  observe it on a finite subset $\Delta \subset V$.
Apart from some special cases, we are not able to show this in general. On the other hand, interestingly, we are able to extend Theorem \ref{thm spin flip 3} when restricting to observations at discrete times. For this, denote by 
\begin{align}\label{eq notation Z}
\integers_{T}:= \{0,\pm T, \pm 2T,\dots \}, \quad T\in(0,\infty).
\end{align}

\begin{thm}\label{thm spin flip 4}
Let $(\eta_t)$, $\lambda$ and $G$ be as in Theorem \ref{thm spin flip 3}. Let $\Delta\subset V$ be finite and let $x \in \Delta$ be such that \eqref{eq survival2} and \eqref{eq survival} hold.
Then, for each $T \in(0,\infty)$, there exist $\rho= \rho(\lambda,T,\Delta)>0$  such that $(\eta_t)$ projected onto $\Delta \times \integers_T$  stochastically dominates a Bernoulli product measure with parameter $\rho$.
\end{thm}


We end this subsection with a result, Theorem \ref{thm spin-flip 2} below, for the supercritical contact process on $\Latd$. As seen in Proposition \ref{prop amenable}, this process cannot stochastically dominate a non-trivial independent spin-flip process, not even when projected onto a subset $\Delta$ of positive density. This naturally leads to the question what happens for subsets $\Delta \subset \Latd$ for which the l.h.s.\ of \eqref{eq density of set} equals $0$.
Theorem \ref{thm spin-flip 2}  concerns one such case, namely, the contact process  projected onto certain (discrete-time) space-time slabs.\medskip

For $m\in \nat$, let
\begin{align}
\integers^d_{d-1}(m):= \left\{ (x_1,\dots, x_d) \in \Latd \colon x_d\in \{0,\dots,m-1\}  \right\},
\end{align}
be the $(d-1)$-dimensional sublattice of $\Latd$ of width $m$. When $m=1$ we simply write $\integers^d_{d-1}$.
\begin{thm}\label{thm spin-flip 2}
Let $(\eta_t)$ be the upper stationary contact process on $\mathbb{Z}^d$, $d\geq 1$, with $\lambda> \lambda_c$. Let  $T \in (0,\infty)$ and $m\in \nat$. Then there exists $\rho=\rho(\lambda,T,m) >0$  such that $(\eta_t)$ projected onto $\integers^d_{d-1}(m) \times \integers_T$  stochastically dominates a Bernoulli product measure with parameter $\rho$. 
\end{thm}

\subsection{Mixing properties}\label{sec mixing}

The purpose of this subsection is to show that the domination results we have obtained so far are useful in order to conclude mixing properties for the contact process, in particular when observed in a subspace.\medskip

We first note that, from the statement of Theorem \ref{thm spin-flip 2} with $m=1$, we obtain a stronger notion of domination, which we present next. For $t\in (0,\infty)$ and $T \in (0,\infty)$, let $\integers_{T}(t) := \{s \in \integers_{T} \colon s < tT \}$ and  denote by 
$\mathcal{P}_{\lambda}^{\text{slab}} \left( \cdot \right)$ the law of the projection of $(\eta_t)$ onto $\Latd_{d-1} \times \integers_{T}.$

\begin{cor}\label{cor spin-flip 2}
Let $(\eta_t)$ be the upper stationary contact process on $\mathbb{Z}^d$, $d\geq 1$, with $\lambda> \lambda_c$. Let  $T \in (0,\infty)$. Then, with $\rho=\rho(\lambda,T,1)$ as in Theorem \ref{thm spin-flip 2}, for every finite $\Delta \subset \Latd_{d-1} \times \integers_{T}(0)$, the measure $\mathcal{P}_{\lambda}^{\text{slab}} \left( \cdot \mid \eta \equiv 0 \text{ on } \Delta   \right)$ stochastically dominates a Bernoulli product measure with density $\rho$ on $\Latd_{d-1} \times \left(\integers_{T} \setminus \integers_{T}(0) \right)$.
\end{cor}

Corollary \ref{cor spin-flip 2} implies that the contact process projected on $\Latd_{d-1} \times \integers_{T}$  has strong mixing properties. We next make precise what we mean by strong mixing properties.\medskip

Fix $ T \in (0,\infty)$  and let, for $\theta \in (0,\frac{1}{2} \pi)$ and $t \geq 0$,
\begin{align}
C_t^{\theta} := \left\{ (x,s) \in \Latd_{d-1} \times \integers_{T} \colon \norm{x} \leq (s-t)  \tan \theta \right\}
\end{align}
be the cone whose tip is at $(o,t)$ and whose wedge opens up with angle $\theta$, where $o\in \Latd$ denotes the origin. A process $(\xi_t)_{t\in \integers_T}$ on $\{0,1\}^{\Latd_{d-1}}$ is said to be \emph{cone-mixing} if, for all $\theta \in (0,\frac{1}{2}\pi)$,
\begin{align}\label{eq cone-mixing}
\lim_{t \rightarrow \infty} \sup_{\substack{A \in \mathcal{F}_{<0}, B \in \mathcal{F}_t^{\theta} \\ \bP(A)>0}} \left| \bP(B\mid A) - \bP(B) \right| = 0,
\end{align}
where $\mathcal{F}_{<0}$  is the $\sigma$-algebra generated by the lower half-space $ \{ \xi_s(x)\colon (x,s) \in \Latd_{d-1}\times \integers_{T}(0) \}$ and $\mathcal{F}_t^{\theta}$ is the $\sigma$-algebra generated by $ \{ \xi_s(x) \colon (x,s) \in C_t^{\theta} \}$. 

\begin{thm}\label{prop cone-mixing}
Let $T \in (0,\infty)$.
The upper stationary contact process  on $\Latd$,  $d\geq 1$, with $\lambda>\lambda_c$, projected onto $\Latd_{d-1}\times \integers_{T}$, is cone-mixing.
\end{thm}

Cone-mixing was introduced in \citet{CometsZeitouniLLNforRWME2004} and used there to prove limiting properties for certain random walks in mixing random environment. More recently, the cone-mixing condition has been adapted to random walks in dynamically evolving random environments, see Avena, den Hollander and Redig \cite{AvenaHollanderRedigRWDRELLN2011}. 
For such models, a standing challenge is to prove limit properties for the random walk when the dynamic environment does not converge towards a unique stationary distribution, uniformly with respect to the initial state. \medskip

Theorem \ref{prop cone-mixing} gives one way to overcome this challenge for the particular case
where the random environment is the contact process and the random walk stays inside $\Latd_{d-1}$.  Our result has recently been applied in \citet{BethuelsenVolleringRWDRE2016} (see Theorem 2.6 therein)  to prove (among other things) a law of large numbers for such random walks.

\section{Preliminaries}\label{sec preliminaries}

In this section we provide some preliminary results which are important for the proofs of our theorems. 

\subsection{Downward FKG and related properties}\label{sec DFKG}

As already mentioned, the concept of downward FKG (from now on abbreviated by dFKG) plays a key role in the proof of our main theorems. We next provide a definition of this  and some related properties.

\begin{defn}\label{def FKG}
Let $\mu\in \mathcal{M}_1(\Omega)$. We say that $\mu$ is
\begin{description}
\item[a)] positively associated if  $\mu(B_1 \cap B_2) \geq \mu(B_1) \mu(B_2)$ for any two increasing events $B_1,B_2 \in \mathcal{F}$.
\item[b)] dFKG if for every finite $\Lambda\subset V$, the measure $\mu(\cdot \mid \eta \equiv 0 \text{ on } \Lambda)$  is positively associated.
\item[c)] FKG if for every finite $\Lambda \subset V$ and $\sigma \in \Omega$, the measure $\mu(\cdot \mid \eta \equiv \sigma \text{ on } \Lambda )$ is positively associated.
\end{description}
\end{defn}

It is immediate that FKG implies  dFKG, which again implies positive association. The Bernoulli product measures $\mu_{\rho}$, $\rho \in [0,1]$, are examples of measures which clearly satisfy the FKG property. In  \cite{LiggettIPS1994} it was shown that the upper invariant measure is not always FKG,
       whereas  \cite{BergHaggstromKahn2006} proved that it satisfies the dFKG property (see Theorem
       3.3 and Equation (20) in  that paper). With the same arguments as in  \cite{BergHaggstromKahn2006} the
       latter property can be extended to the following lemma.

\begin{lem}\label{lem DFKG for CP}
Consider the upper stationary contact process $(\eta_t)$ on $G=(V,E)$ with $\lambda>0$.
For any $t_1<t_2< \dots < t_n$ the joint distribution of $(\eta_{t_1},\dots,\eta_{t_n})$, which is a probability measure on $\Omega^n$,
satisfies the dFKG property.
\end{lem}
\begin{proof}
The proof is exactly the same as the proof of Theorem 3.3 in \cite{BergHaggstromKahn2006}.
\end{proof}

The following lemma gives a useful property, used in the proof of Theorem \ref{thm spin flip 4}.

\begin{lem}\label{lem dfkg max}
Let $V$ be countable and assume that the random variables $(X_i)_{i \in V}$ are dFKG. Let $P=(P_j)_{j \geq 1}$ be a partitioning of $V$ into disjoint subsets. Then the random variables $(Y_j)_{j\geq1}$ where
$Y_j = \max \{ X_i, \: i \in P_j \}$ 
are dFKG.
\end{lem}

\begin{proof}
This follows easily from the dFKG property of $(X_i)$. (Use that the $Y_j$'s are increasing functions of $(X_i)$ and that $\{ Y_j =0\} = \{ X_i=0, i \in P_j\}$).
\end{proof}

The dFKG property was used in \cite{LiggettSteifSD2006} to give a sufficient and necessary condition for a translation invariant measure $\mu$ on $\{0,1\}^{\integers}$ to dominate a Bernoulli product measure with density $\rho\in [0,1]$. Since their result plays an important role for our proofs, we recall the precise statement. 

\begin{lem}[Theorem 1.2 in \cite{LiggettSteifSD2006}]\label{lem DFKG domi1}
Let $V= \integers$ and let $\mu \in \mathcal{M}_1(\Omega)$ be a translation invariant measure on $\{0,1\}^{\integers}$ which is dFKG.  Then the following are equivalent. 
\begin{enumerate}
\item $\mu$ stochastically dominates $\mu_{\rho}$.
\item $\mu( \eta \equiv 0 \text{ on } \{1,2,\dots, n\} ) \leq (1-\rho)^n$ for all $n$.
\item For all disjoint, finite subsets $\Lambda$ and $\Delta$ of $\{1,2,3,\dots\}$, we have
\begin{align}
\mu \left( \eta(0)=1 \mid \eta \equiv 0 \text{ on } \Lambda, \eta \equiv 1 \text{ on } \Delta \right) \geq \rho.
\end{align}
\end{enumerate}
\end{lem}

In \cite{LiggettSteifSD2006} also a generalisation of Lemma \ref{lem DFKG domi1} to measures on $\{0,1\}^{\Latd}$ with $d\geq2$ is presented. Though most of our arguments only use Lemma \ref{lem DFKG domi1}, for the proof of Corollary \ref{cor spin-flip 2} we need the higher dimensional version, which we state below. We use the notation
\[ \mathcal{D} :=\left\{ (x_1,\dots, x_d) \in \Latd \colon \exists m \text{ such that } x_i = 0\: \forall i<m \text{ and } x_m <0\right\}.\]

\begin{lem}[Theorem 4.1 in \cite{LiggettSteifSD2006}]\label{lem DFKG domi}
Let $V= \Latd$ with $d \geq 2$ and let $\mu \in \mathcal{M}_1(\Omega)$ be a translation invariant measure on $\{0,1\}^{\Latd}$ which is dFKG.  Then the following are equivalent. 
\begin{enumerate}
\item $\mu$ stochastically dominates $\mu_{\rho}$.
\item $\mu( \eta \equiv 0 \text{ on } [1,n]^d ) \leq (1-\rho)^{n^d}$ for all $n$.
\item For all disjoint, finite subsets $\Lambda$ and $\Delta$ of $\mathcal{D}$, we have
\begin{align}
\mu \left( \eta(o)=1 \mid \eta \equiv 0 \text{ on } \Lambda, \eta \equiv 1 \text{ on } \Delta \right) \geq \rho.
\end{align}
\end{enumerate}
\end{lem}

\begin{rem}
Lemma \ref{lem DFKG domi} was stated (and proven) in \cite{LiggettSteifSD2006} for $d=2$. However,  the extension of their argument to general dimensions is immediate and yields Lemma \ref{lem DFKG domi} (as also commented directly before the proof in \cite{LiggettSteifSD2006}, see p.\ 232 therein).
\end{rem}

\subsection{The contact process}\label{sec contact}

 We next give a brief and somewhat informal construction of the contact process via the so-called graphical representation.
 For a more thorough  description we refer to \cite{LiggettSIS1999}, p.\ 32-34. \medskip

Let $G=(V,E)$ be a connected graph having bounded degree and fix $\lambda \in (0,\infty)$. 
 Let $H:= (H(x))_{x \in V}$ and $I:=(I(x,y))_{\{x,y\}\in E}$ be two independent collections of (doubly-infinite) i.i.d Poisson processes with rate $1$ and $\lambda$, respectively.
On $V \times \reals$, draw the events of $H(x)$ as \emph{crosses} over $x$ and the events of $I(x,y)$ as \emph{arrows} from $x$ to $y$.
\medskip

For $x,y \in V$ and $ s \leq t$, we say that $(y,t)$ is connected to $(x,s)$ by a backwards path, written $(x,s) \leftarrow (y,t)$, if and only if there exists a directed path in $V \times \reals$ starting at $(y,t)$, ending at $(x,s)$ and going either backwards in time without hitting crosses or ``sideways'' following arrows in the opposite direction of the prescribed direction. Otherwise we write $(x,s) \nleftarrow (y,t)$. In general, for $\Lambda, \Delta \subset V \times \reals$, we write $\Delta \leftarrow \Lambda$ ($\Delta \nleftarrow \Lambda$) if there is a (there is no) backwards-path from $\Lambda$ to $\Delta$. Next, define the process $(\tilde{\eta}_t)$ on $\Omega$ by
\begin{align}
\tilde{\eta}_t(x) := \left\{\begin{array}{cc}1, &  \text{ if }V \times \{-\infty \} \leftarrow (x,t); \\0, & \text{otherwise},\end{array}\right.
\end{align}
where $V \times \{-\infty \} \leftarrow (x,t)$ denotes the event that there exists a backwards-path from $(x,t)$ to $V \times \{s\}$ for all $s\leq t$.
It is well known that $(\tilde{\eta}_t)$ has the same distribution as the upper stationary contact process $(\eta_t)$ with infection parameter $\lambda>0$. In the following we use  the notation $(\eta_t)$ for either representations of the contact process and denote by $\mathcal{P}_{\lambda}$ the corresponding path measure. 
\medskip

We next state a lemma which is useful for most of our proofs. The proof and the statement is inspired by \cite[Lemma 2.11]{BirknerCernyDepperschmidtRWDRE2015}. 
For its statement, recall \eqref{eq results survival} and note that, 
as follows from the graphical representation,
\begin{align}
\cP_{\lambda}\left( s<\tau^x <\infty\right) =\cP_{\lambda} \left( V \times \{-s\} \leftarrow (x,0) \text{ but } V \times \{-\infty \} \nleftarrow (x,0) \right).
\end{align}

\begin{lem}\label{lem spin-flip 1}
Consider the upper stationary contact process on a connected graph $G=(V,E)$ of bounded degree with $\lambda>0$. Let $\Delta \subset V$ and assume that there exist $\epsilon, C,c>0$ such that for all $x \in \Delta$,
\begin{align}\label{eq lem spin-flip 1 assump2}
 &\mathcal{P}_{\lambda} \left( \tau^x=\infty \right) >\epsilon;
\\ &\label{eq lem spin-flip 1 assump}
\mathcal{P}_{\lambda} \left( s<\tau^x <\infty
\right) \leq C e^{-cs}, \quad s \geq 0.
\end{align}

Then, for any $T \in (0,\infty)$, there exists $\rho=\rho(T)>0$ such that for all $n$ and all $x_1,\dots, x_n \in \Delta$; 
\begin{align}\label{eq lem spin-flip 1}
\mathcal{P}_{\lambda} \left( \eta_{Ti}(x_i)=0, i = 1,2,\dots, n\right) \leq (1-\rho)^n.
\end{align} 
\end{lem}

\begin{proof}
Fix $T \in (0,\infty)$ and let $\textbf{x}=(x_i)_{i \in \integers}$ be an infinite sequence of elements $x_i \in \Delta$. For $i \in \bZ$, denote by \begin{align}D_i := \inf \{ l \in \nat \colon V \times T(i-l) \nleftarrow (x_i,T i) \},\end{align} and note that $D_i T$ yields an approximation (up to an error of at most $T$) on how far backwards in time $(x_i,T i)$ is connected to another space-time point. In particular, $\eta_{T i}(x_i) = 0$ if and only if $D_i<\infty$. 

Define $\mathcal{T}_0=0$ and, iteratively,
\begin{align}
\cT_{i+1} := \cT_i + D_{n-\cT_i}, \quad i \geq 0.
\end{align}
Let $K:= \sup \{i \colon \cT_i<\infty \}$. 
We have the following relation (easy to check) between events:
\begin{align}
\{ \eta_{iT}(x_i) =0 \text{ for } i=1,\dots,n \} &= \{ D_1,\dots,D_n <\infty \}
\\ &\subset \{\mathcal{T}_K \geq n\}.
\end{align}

Finally, $\mathcal{P}_{\lambda} (\mathcal{T}_K\geq n)$ is exponentially small in $n$. This follows by standard arguments from
the following consequences of \eqref{eq lem spin-flip 1 assump2} and \eqref{eq lem spin-flip 1 assump} (using the independence properties of the graphical representation):
for all $i$, all positive integers $t_i>t_{i-1}> \dots >t_1 \geq 1$, and all $s \geq 1$, we have
\begin{align}
&\mathcal{P}_{\lambda} \left( D_{n-\mathcal{T}_i} = \infty \mid \cT_1=t_1, \dots, \cT_i=t_i \right) >\epsilon;
\\ &\mathcal{P}_{\lambda} \left( s \leq D_{n-\mathcal{T}_i} < \infty \mid \cT_1=t_1, \dots, \cT_i=t_i \right) \leq Ce^{-c(s-1)}.
\end{align} \end{proof}



\section{Proofs}\label{sec proofs}

\subsection{Proof of Proposition \ref{prop amenable}}\label{sec proofs 1}

For the proof of Proposition \ref{prop amenable} we follow that of \cite[Proposition 1.1]{LiggettSteifSD2006}, which we extend to graphs having subexponential growth.

\begin{proof}[Proof of Proposition \ref{prop amenable}]

Let $G=(V,E)$ be a graph as in the statement of the proposition and let $\lambda \in (0, \infty)$. Fix $o \in V$, and consider $\Delta \subset  V$ having positive density. Hence, there is a $\gamma>0$ and a $N\in \nat$ such that, for all $n \geq N$, we have that $|\Delta \cap B(n)|>\gamma |B(n)|$, where $B(n) := \{ x\in V \colon d(o,x) \leq n\}$. 

Next, assume that the contact process on $G$ with infection parameter $\lambda>0$, projected onto $\Delta$, stochastically dominates a non-trivial independent spin-flip process with parameter $\alpha>0$. 
Consequently, for every $T>0$ and $n\geq N$, we have that
\begin{align}\label{eq prop amenable bound}
&\mathcal{P}_{\lambda} \left( \eta_t(x) =0 \text{ for all } (x,t) \in B(n)\times[0,T] \right) 
\\ \leq & \mathcal{P}_{\lambda} \left( \eta_t(x) =0 \text{ for all } (x,t) \in (\Delta \cap B(n)) \times[0,T] \right)
\\ \leq &e^{-\gamma \alpha|B(n)| T} = e^{-c_1|B(n)| T}, \quad \text{ where  }c_1 = \gamma \alpha.
\end{align}
Thus, the probability in \eqref{eq prop amenable bound} decays exponentially at a rate proportional to the volume of $B(n)\times[0,T]$.

To conclude the statement of Proposition \ref{prop amenable} for $\lambda>0$, we show that this estimate cannot hold and thus argue  by means of contradiction. In doing so, we make use of the graphical representation of the contact process. 

Let $A_{n,T}$ denote the event that there are no arrows in the graphical representation from sites outside $B(n)$ to any site in $B(n)$ during the time period $[0, T ]$.  Note that the l.h.s.\ of  \eqref{eq prop amenable bound} is bounded below by
\begin{align}
\mathcal{P}_{\lambda} \left( \{ \eta_0(x) = 0 \text{ for } x \in B(n) \} \cap A_{n,T} \right).
\end{align}
Moreover, this is again bounded below by
\begin{align}
&\mathcal{P}_{\lambda} \left( \{ \eta_0(x) = 0 \text{ for } x \in B(n) \} \right) e^{-\lambda d | B(n+1) \setminus B(n)| T}
\\ \label{eq prop am bound 2}\geq &\left[ \prod_{x \in B({n})}\bar{\nu}_{\lambda}(\eta_0(x) = 0 ) \right]e^{-\lambda d | B(n+1) \setminus B(n)| T},
\end{align} 
where $d$ denotes the maximum degree of $G$, and where we used that the contact process is positively associated.

Next, since $G$ has subexponential growth (and hence satisfies \eqref{eq extend subexp}), we can find $n$ large such that $\lambda d |B(n+1)\setminus B(n)| < c_1 | B(n)|$. 
For such $n$, by taking $T$ sufficiently large, the expression \eqref{eq prop am bound 2} is larger than the r.h.s.\ of \eqref{eq prop amenable bound}: a contradiction.
\end{proof}

\subsection{Proof of Theorem \ref{thm spin flip 3}}\label{sec proofs 3}
\begin{proof}[Proof of Theorem \ref{thm spin flip 3}]
Consider the upper stationary contact process $(\eta_t)$ on a connected graph $G=(V,E)$ having bounded degree and with $\lambda>0$.
Fix $x\in V$ such that  \eqref{eq survival2} and \eqref{eq survival} hold and define, for $t,s\in \reals$ with $t<s$, the event $A_{t,s}:= \{ \eta_u(x)=0 \colon u\in[t,s)\}$. 
Further, let $f\colon [0,\infty]\times[0,\infty)\rightarrow [0,1]$ denote the function
\begin{align}
f(t,u) = \cP_{\lambda} \left( A_{0,t} \mid A_{-u,0} \right).
\end{align}
Clearly, $f(t,u)$ is non-increasing in $t$.

By Lemma \ref{lem DFKG for CP} we have that, for each $n$, the collection of random variables $\left(\eta_t(y), y \in V, t\in \integers_{1/n} \right)$ is dFKG (recall from \eqref{eq notation Z} that $\integers_{1/n}$ denotes $\{ k/n \colon k \in \integers \}$).
Further, it is standard (and easy to see) that, for $t<s$,
\begin{align}
\cP_{\lambda} (A_{t,s}) = \lim_{n \rightarrow \infty} \cP_{\lambda} \left( \eta_u(x)=0 \text{ for all } u \in [t,s) \cap \integers_{1/n} \right).
\end{align}
Using this approximation, the above mentioned dFKG property, and general results for measures satisfying dFKG (see Section \ref{sec DFKG}), it follows that 
\begin{align}\label{eq one star}
f(t,u) \text{ is non-decreasing in } u,
\end{align}
so $f(t):=\lim_{u \rightarrow \infty} f(t,u)$ exists (and is $>0$) and $\cP_{\lambda}(A_{0,t}\mid B) \leq f(t)$ for all events $B$ that are measurable with respect to $(\eta_s(x), s \leq 0)$.
Further, since,
\begin{align}
f(t+s,u) &= \cP_{\lambda} (A_{0,t+s} \mid A_{-u,0})
\\ &= \cP_{\lambda}(A_{0,t} \mid A_{-u,0}) \cP_{\lambda}(A_{t,t+s} \mid A_{-u,t})
\\ &= f(t,u) f(s,t+u),
\end{align}
we get, by letting $u\rightarrow \infty$,
 $f(t+s)=f(t)f(s)$, 
from which we obtain that there is a $c\geq0$ such that 
\begin{align}\label{eq two stars}
f(t) = e^{-ct}, \quad \text{for all } t \geq 0.
\end{align}
By Lemma \ref{lem spin-flip 1} (with $T=1$), there is an $\alpha>0$ such that
\begin{align}\label{eq six stars}
\cP_{\lambda}(A_{0,t}) \leq e^{-\alpha t}, \quad t \geq 1.
\end{align}
We claim that $c \geq \alpha$ (and hence $c>0$).
The proof of this claim uses some of the arguments in the proof of Lemma \ref{lem DFKG domi1} in \cite{LiggettSteifSD2006}. For completeness, we include it here.

Suppose $c<\alpha$. Let $\alpha' \in (c,\alpha)$. Fix $t>1$ and take an integer $l$ so large that $f(t,lt)$ is `very close' to $f(t)$ (and hence, by \eqref{eq two stars}, to $e^{-ct}$). More precisely, we take $l$ sufficiently large so that 
\begin{align}\label{eq three stars}
f(t,lt)> e^{-\alpha' t}.
\end{align}
For all integers $k \geq 0$ we have that, on the one hand (by \eqref{eq six stars}),
\begin{align}\label{eq four stars}
\cP_{\lambda}(A_{0,klt}) \leq e^{-\alpha klt},
\end{align}
while on the other hand
\begin{align}
\cP_{\lambda}(A_{0,klt}) &= \cP_{\lambda}(A_{0,lt}) \prod_{i=l}^{kl-1} \cP_{\lambda} \left( A_{it,(i+1)t} \mid A_{0,it}\right)
 \\&\geq \cP_{\lambda}(A_{0,lt}) (f(t,lt))^{kl}
\\&>  \cP_{\lambda}(A_{0,lt})  e^{-\alpha'tkl},
\end{align}
where the first inequality uses \eqref{eq one star} and stationarity, and the second inequality comes from \eqref{eq three stars}.
Since $\alpha'<\alpha$ (and $\cP_{\lambda}(A_{0,lt})>0$) this violates \eqref{eq four stars} if $k$ is sufficiently large, and yields a contradiction. This proves the claim.

By the claim, and the inequality one line below \eqref{eq one star}, we have that $\cP_{\lambda}(A_{0,t} \mid B) \leq e^{-\alpha t}$ for all events $B$ that are measurable with respect to $(\eta_s(x), s\leq 0)$.

Finally, we also clearly have (by the contact process dynamics) that the conditional probability of the event $\{\eta_s(x)=1 \text{ for all } s \in (0,t)\}$, given that $\eta_0(x)=1$ and any additional information about the process before time $0$, is exactly $e^{-t}$. We conclude that the process $(\eta_s(x))$ dominates a spin-flip process which goes from state $0$ to $1$ at rate $\alpha$ and from $1$ to $0$ at rate $1$.

\end{proof}

\subsection{Proof of Theorem \ref{thm spin flip 4}}
\begin{proof}[Proof of Theorem \ref{thm spin flip 4}]
Fix $T \in (0,\infty)$ and let $\Delta \subset V$ be finite with $x \in \Delta$  such that \eqref{eq survival2} and \eqref{eq survival} hold. Furthermore, consider the doubly infinite sequence $(Y_i)_{i \in \integers}$, where $Y_i$ is given by
\begin{align}\label{eq proof thm sf 4 def}
Y_i := \max \{ \eta_{T i}(y) \colon y \in \Delta \}.
\end{align}
By Lemma \ref{lem DFKG for CP} and Lemma \ref{lem dfkg max} we note that $(Y_i)$ is dFKG, and, since the upper stationary contact process is invariant under temporal shift, the sequence is also translation invariant. By Lemma \ref{lem spin-flip 1}, there is a $\rho>0$ such that
\begin{align}
\cP_{\lambda} \left( Y_j=0, j=1,\dots n \right) \leq (1-\rho)^n.
\end{align}
Hence, by Lemma \ref{lem DFKG domi1}, we get 
\begin{align}\label{eq proof thm sf 4}
&\mathcal{P}_{\lambda} \left( Y_1=1 \mid Y_{-j}=0, j=0,\dots,n\right)\geq \rho.
\end{align}
It is not difficult to see that \eqref{eq proof thm sf 4} yields the following: for some $0< \tilde{\rho} \leq \rho$, 
\begin{align}\label{eq strong domi}
\mathcal{P}_{\lambda} \left(\eta_{T}(x) = 1 \text{ for all }x \in \Delta \mid Y_{-j}=0, j =0,\dots, n \right) \geq \tilde{\rho}, 
\end{align}
 for all  $n \in \nat$. Indeed, since the contact process evolves  in continuous-time and the graph is connected, infections can spread with positive probability from any point in $\Delta$ to all other points in $\Delta$ in a small time interval.

To make this more formal one can first consider a sequence defined similar to $(Y_i)$,  only replacing $T$ by $T/2$ in  \eqref{eq proof thm sf 4 def}.
By the same argument as above, using again the dFKG property, we have that for some $\delta>0$,
\begin{align}
\mathcal{P}_{\lambda}\left( \max \{ \eta_{T/2}(y)\colon y \in \Delta \} =1 \mid Y_{-j}=0, j=0, \dots, n\right) >\delta.
\end{align}
Furthermore, since $\Delta$ is finite, and $G$ is connected, there is an $\epsilon>0$ such that (with the notation introduced below \eqref{eq results survival})
\begin{align}
\inf_{z \in \Delta} \mathcal{P}_{\lambda} \left( \eta_{\frac{T}{2}}^z(y) = 1 \text{ for all } y \in \Delta \right) >\epsilon.
\end{align} 
Thus, using the fact that the contact process is a Markov process, we conclude \eqref{eq strong domi} with $\tilde{\rho}\geq \epsilon \delta >0$.

Finally, using again the dFKG property of the collection $( \eta_{Ti}(y), y \in \Delta, i \in \integers)$, we obtain that \eqref{eq strong domi} still holds if  the conditioning $\{Y_{-j}=0, j =0,\dots, n\}$ is replaced by any event measurable  with respect to $( \eta_{-Ti}(y), y \in \Delta, i \geq 0 )$. This concludes the proof of the theorem.
\end{proof}

\subsection{Proof of Theorem \ref{thm spin-flip tree}}\label{sec proof of spt}

To prove Theorem \ref{thm spin-flip tree} we first prove that the contact process on $\{0,1,\dots\}$  observed at the vertex $\{0\}$ stochastically dominates an independent spin-flip process. Indeed, the required estimates \eqref{eq survival2}  and \eqref{eq survival} for this context is provided by the following result in \cite{DurrettGriffeathCPhighDim1982}, see Equation (21) on page 546 therein.

\begin{lem}[\cite{DurrettGriffeathCPhighDim1982}, Equation (21), and \cite{DurretGriffeathCP1983}] \label{lem contact N}
Consider the contact process on $V= \{0,1,\dots \}$ with $\lambda>\lambda_c$. Then there exists constants $\epsilon, C,c>0$ such that \eqref{eq survival2} and \eqref{eq survival} hold.
\end{lem}

\begin{proof}[Proof of Theorem \ref{thm spin-flip tree}]

Firstly, by Lemma \ref{lem contact N} applied to Theorem \ref{thm spin flip 3}, we have that the contact process on $\{0,1,2,\dots\}$ with $\lambda>\lambda_c$ observed at the vertex  $\{0\}$ stochastically dominates an independent spin flip process with $\alpha>0$. 

 From the above observation, the statement of Theorem \ref{thm spin-flip tree} follows by a monotonicity argument using again the graphical construction of the contact process. 
 
To make this last argument precise, fix an arbitrary point $o\in T_d$ and call it the root. Denote by $u(o)=0$ its label. Furthermore, label the remaining sites according to their distance with respect to $o$ in a unique way. That is, each $x\in T_d$ with $\norm{x-o}=1$ has a label $u(x)=(0,i)$ for some $i\in\{1,\dots,d+1\}$ and for $y \in T_d$ satisfying $\norm{y-o}=\norm{z-o}+1=n$ and $\norm{y-z}<n$, $n\geq 2$, set
$u(y)=(u(z),i)$, $i \in \{1,\dots,d\}. $ 
Thus, for each $x\in T_d\setminus \{o\}$, we have that \[u(x)\in \bigcup_{n \geq 0}\left[\{0\} \times  \{1,\dots,d+1\}  \times \{1,\dots, d\}^n\right]. \]




Denote by $\Delta\subset T_d$ the set of vertices  having as  last entry of its label a number different from $1$.
Using the graphical representation of the contact process, consider the process $(\xi_t)$ on $T_d$ where for each $(x,t) \in \Delta \times \reals$ we set $\xi_t(x)=1$  if and only if there is an infinite backwards path from $(x,t)$ constrained to infection arrows between the sites with label $\{u(x),(u(x),1),(u(x),1,1),\dots \}$. Moreover, for $x\in \Delta^c$, let $\xi_t(x)=0$ for all $t \in \reals$.

By construction, the evolution of $(\xi_t)$ on $T_d$ is dominated by that of the contact process. Furthermore, the evolution at site $x\in \Delta$ is in one-to-one correspondence with the contact process on $\{0,1,2,\dots\}$, and the evolution at different sites $x,y\in \Delta$ is independent. Thus, on the set $\Delta$ the process $(\xi_t)$ stochastically dominates a non-trivial independent spin-flip process, and consequently, so does also the contact process. Lastly, we note that, from the above construction, it holds that $\Delta$ has positive density and that the l.h.s.\ of \eqref{eq density of set} equals $\gamma= \frac{d-1}{d}>0$. This concludes the proof.

\end{proof}

\subsection{Proof of Theorem \ref{thm spin-flip 2}}\label{sec proof spin-flip 2}

In order to prove Theorem \ref{thm spin-flip 2}, we make use of the well known fact that the supercritical contact process on $\Latd$ with $d\geq 2$ survives in $2$-dimensional space-time slabs (see \cite{BezuidenhoutGrimmettCP1990}). More precisely, let,  for $k \in \nat$,  
\begin{align}
S_{k} := \left\{ x \in \Latd \colon x_i \in \{0,1,\dots,k-1\}, i=1,\dots d-1 \right\},
\end{align}
and denote by $(_k\eta_t)$  the contact process on $S_k$ and by $\mathcal{P}_{\lambda,k}$ its path measure. 
This process on $S_k$ with $\lambda>\lambda_c(\Latd)$ survives with positive probability if the width $k$ is large enough. 
The proof of this proceeds via a block argument and comparison with a certain $2$-dimensional (dependent) directed percolation model. 
This argument also gives a form of exponential decay, more precisely, the following lemma holds.
 
\begin{lem}\label{lem slabs}
Let $\lambda >\lambda_c(\Latd)$ and $d\geq2$. Then there exists $k\in \nat$  and $\epsilon, C,c \in (0,\infty)$, such that for all  $x \in S_k$,
\begin{align}
\label{eq slab2}
&\cP_{\lambda,k} \left( \tau^x=\infty \right)>\epsilon;
\\ &\label{eq slab}
\cP_{\lambda,k} \left( s<\tau^x<\infty \right) \leq Ce^{-cs}, \quad \text{ for all } s>0.
\end{align}
\end{lem}

\begin{proof}
This follows again by comparison with a $2$-dimensional directed percolation model and a renormalization arguments. For a proof 
we refer the reader to the proof of Theorem 1.2.30a) in \cite{LiggettSIS1999}, where such an argument is explained in detail. Though proved there for the unrestricted contact process $(\eta_t)$ the argument works, mutatis mutandis, for $(_k\eta_t)$ as soon as $k$ is taken sufficiently large.
\end{proof}

\begin{proof}[Proof of Theorem \ref{thm spin-flip 2}]
Fix $T \in (0,\infty)$ and note that the case $d=1$ is an immediate consequence of Theorem \ref{thm spin flip 4}. Indeed, the estimates  \eqref{eq survival2} and \eqref{eq survival} for that case are known to hold due to \cite[Theorem 5]{DurretGriffeathCP1983}.

For the case $d\geq 2$ we use a slightly more involved argument, by partitioning $\Latd \times \reals$ into slabs. Fix $k$ such that \eqref{eq slab2} and \eqref{eq slab} hold. For $\textbf{i}=(i_1,\dots,i_{d-1})\in \integers^{d-1}$, let $P_{\textbf{i}}= \left( S_{k} + k \cdot (i_1,\dots,i_{d-1},0)\right) \times \reals$. Note that $P_{\textbf{i}} \cap P_{\textbf{j}}=\emptyset$ whenever $\textbf{i}\neq \textbf{j}$ and that $\bigcup_{\textbf{i} \in \integers^{d-1}} P_{\textbf{i}} = \Latd \times \bR$.

Next, consider the process $(\zeta_t)$ which is obtained from the graphical representation of the contact process on $\Latd$ by suppressing all infection arrows between slabs $P_{\textbf{j}}$.
Trivially the evolution of $(\zeta_t)$ is dominated by that of $(\eta_t)$.
Moreover, the evolution of $(\zeta_t)$ in each slab is independent of the others and has the same law as $(_k\eta_t)$.

Let $\textbf{i} \in \integers^{d-1}$. By applying Theorem \ref{thm spin flip 4} with 
$\Delta = \Latd_{d-1}(m) \cap (S_k +k\cdot (\textbf{i},o))$,  
it follows that the process $(\zeta_t)$ observed on the vertices $\Delta$ at times that are multiples of $T$ stochastically dominates a non-trivial Bernoulli product measure with density $\rho>0$. By the above mentioned independence, this implies the statement of Theorem \ref{thm spin-flip 2} for $(\zeta_t)$.
Since  $(\zeta_t)$ is stochastically dominated by $(\eta_t)$, we conclude the proof.
\end{proof}

\subsection{Proof of Corollary \ref{cor spin-flip 2}}

\begin{proof}[Proof of Corollary \ref{cor spin-flip 2}]
Fix $T \in (0,\infty)$ and recall the definition of $\cP_{\lambda}^{\text{slab}}$ in Section \ref{sec mixing}. 
Note that $\cP_{\lambda}^{\text{slab}}$ is translation invariant and that, due to Lemma \ref{lem DFKG for CP}, it is also dFKG. In particular, we may apply Lemma \ref{lem DFKG domi} to $\cP_{\lambda}^{\text{slab}}$.

 A direct consequence of Theorem  \ref{thm spin-flip 2} with $m=1$ is that whenever $\lambda>\lambda_c$, there is a $\rho>0$ such that
 \begin{align}
 \cP_{\lambda}^{\text{slab}} \left( \eta_s(x) = 0, (x,s) \in \Latd_{d-1}\times \integers_{T} \cap [1,n]^d \times [T,nT] \right) \leq (1-\rho)^{n^d}.
 \end{align}
 Hence, the measure $\cP_{\lambda}^{\text{slab}}$ satisfies Property $2$ in Lemma \ref{lem DFKG domi}. Consequently,  $\cP_{\lambda}^{\text{slab}}$ also satisfies Property $3$ in Lemma \ref{lem DFKG domi}, from which the statement of Corollary \ref{cor spin-flip 2} follows.
\end{proof}

\subsection{Proof of Theorem \ref{prop cone-mixing}}

Theorem \ref{prop cone-mixing}  follows from Corollary \ref{cor spin-flip 2} and a standard coupling argument, together with classical properties of the contact process.

\begin{proof}[Proof of Theorem \ref{prop cone-mixing}]

Fix $ T \in (0,\infty)$ and let $\rho>0$ be such that the statement of Corollary \ref{cor spin-flip 2} holds.
Next, denote by $\mu \in \mathcal{M}_1(\Omega)$ the probability measure under which all vertices outside $\Latd_{d-1}$ have value $0$ a.s., and those in $\Latd_{d-1}$ correspond with independent Bernoulli random variables with parameter $\rho$. Further, for $\eta \in \Omega$, denote by $\delta_{\eta} \in \mathcal{M}_1(\Omega)$ the probability measure which concentrates on $\eta$, and write $\bar{1}\in \Omega$ for the configuration where all sites are equal to $1$.
Then,  by Corollary \ref{cor spin-flip 2}, and since $\bar{\nu}_{\lambda} \leq \delta_{\bar{1}}$, 
we have, for $\theta \in (0,\pi/2)$, $t >0$ and $B \in \mathcal{F}_t^{\theta}$ increasing, and for any $A \in \mathcal{F}_{<0}$ with $\mathcal{P}_{\lambda}^{\text{slab}}(A)>0$, that
\begin{align}\label{eq thm 1.7 hippo}
 \left| \mathcal{P}_{\lambda}^{\text{slab}}(B\mid A) - \mathcal{P}_{\lambda}^{\text{slab}}(B) \right| 
  &\leq \widehat{\mathcal{P}}_{\mu,\delta_{\bar{1}}} \left( \eta^1 \neq \eta^2 \text{ on } C_t^{\theta} \right),
\end{align} 
where $ \widehat{\mathcal{P}}_{\mu,\delta_{\bar{1}}}$ is the standard graphical construction coupling of the contact processes on $\Latd$ started at time $0$ from a configuration drawn according to $\mu$ and $\delta_{\bar{1}}$, respectively. 

Furthermore, we have that
\begin{align}\label{eq thm 1.7 last one}
\begin{split}
 \widehat{\mathcal{P}}_{\mu,\delta_{\bar{1}}} \left( \eta^1 \neq \eta^2 \text{ on } C_t^{\theta} \right)
  &\leq \sum_{ (x,s) \in C_t^{\theta}} \widehat{\mathcal{P}}_{\mu,\delta_{\bar{1}}} \left( \eta^1_s(x) \neq \eta^2_s(x) \right) 
 \\ &= \sum_{ (x,s) \in C_t^{\theta}} \widehat{\mathcal{P}}_{\mu,\delta_{\bar{1}}} \left( \eta^1_s(o) \neq \eta^2_s(o) \right),
 \end{split}
\end{align}
where the last equation holds due to translation invariance in the first $(d-1)$ spatial directions. 

Since the set of increasing events in  $\mathcal{F}_t^{\theta}$ generates $ \mathcal{F}_t^{\theta}$,  in  order to conclude the argument, it is sufficient to show that, for some $C,c\in(0,\infty)$, we have 
 \begin{align}\label{eq add name}\widehat{\mathcal{P}}_{\mu,\delta_{\bar{1}}} \left( \eta^1_s(o) \neq \eta^2_s(o) \right) \leq Ce^{-cs}.\end{align}
This can be shown using known estimates for the supercritical contact process on $\Latd$. For completeness we present the details.

Let $N:= \inf \{ \norm{x} \colon \eta_0^1(x) = 1 \text{ and } \tau^x=\infty \}$. Then, for any $a>0$, we have that
\begin{align}\label{eq thm 1.7 very last}
\begin{split}
\widehat{\mathcal{P}}_{\mu,\delta_{\bar{1}}} \left( \eta^1_s(o) \neq \eta^2_s(o) \right)  &\leq 
\widehat{\mathcal{P}}_{\mu_{\rho},\delta_{\bar{1}}} \left( \{ N > as \} \right) \\&+ \widehat{\mathcal{P}}_{\mu_{\rho},\delta_{\bar{1}}} \left( \{\eta^1_s(o) \neq \eta^2_s(o) \} \cap \{N\leq as\}\right).
\end{split}
\end{align}
That the first term on the righthand side decays exponentially (in $as$) follows from \cite[Theorem 1.2.30]{LiggettSIS1999}. 
For the other term, we have that
\begin{align}
&\widehat{\mathcal{P}}_{\mu_{\rho},\delta_{\bar{1}}} \left( \{\eta^1_s(o) \neq \eta^2_s(o) \} \cap \{ N\leq as\} \right)
\\ \leq &\sum_{y\in [-as,as]^d} \widehat{\mathcal{P}}_{\delta_{{\bar{0}_y}},\delta_{\bar{1}}} \left( \{\eta^1_s(o) \neq \eta^2_s(o) \} \cap \{\tau^y=\infty \} \right)
\\ \label{eq thm 1.7 latter term} \leq &\sum_{y\in [-as,as]^d} \widehat{\mathcal{P}}_{\delta_{{\bar{0}_y}},\delta_{\bar{1}}} \left( \eta^1_s(o) \neq \eta^2_s(o) \mid \tau^y=\infty  \right),
\end{align}
where $\bar{0}_{y}$ is the configuration given by $\bar{0}_y(x)=\bar{0}(x)=0$ for all $y\neq x$ and $\bar{0}_{y}(y) = 1-\bar{0}(y)=1$. 
From the large deviation estimates obtained in \cite[Theorem 1.4]{GaretMarchandLDPCPRE2013}, by choosing $a>0$  in \eqref{eq thm 1.7 latter term} sufficiently small,   the term inside the sum of  \eqref{eq thm 1.7 latter term} decays exponentially (in $s$), uniformly for $y \in [-as,as]^{d}$. Hence, since the sum only contains polynomially many terms,  we have that  \eqref{eq thm 1.7 very last}  decays exponentially with respect to $s$. 

In conclusion, there exist $C,c>0$ such that \eqref{eq add name} holds, 
from which, by \eqref{eq thm 1.7 hippo} and \eqref{eq thm 1.7 last one}, we conclude the proof.

 \end{proof}


\section{Open questions}\label{sec questions}

We expect that the statement of Theorem \ref{thm spin-flip tree} can be improved. 

\begin{q}
Can the condition $\lambda>\lambda_c(\integers)$ in Theorem \ref{thm spin-flip tree} be replaced by $\lambda>\lambda_c(T_d)$?
\end{q}
\begin{q}
Does Theorem \ref{thm spin-flip tree} hold with $\Delta = T_d$?
\end{q}


Motivated by Theorem \ref{thm spin-flip 2}  and Lemma \ref{lem spin-flip 1}, the following questions seem natural.

\begin{q}\label{q path}
 Consider the upper stationary contact process on $\Latd$,  $d\geq1$, with $\lambda>\lambda_c$, and let $\textbf{x} = (x_i)_{i \in \integers}$ be an infinite sequence of elements in $\Latd$. Does the contact process projected onto $\{ (x_i,i) \colon i \in \integers \}$ stochastically dominate a non-trivial Bernoulli product measure?
\end{q}

\begin{q}\label{q walk}
Consider the upper stationary contact process $(\eta_t)$ on $\Latd$, $d\geq1$, with $\lambda>\lambda_c$, and let $\textbf{X}=(X_i)_{i \geq 0}$ be a simple random walk on $\Latd$ started at $X_0=o$.
 Does the sequence $(\eta_i(X_i))_{i \geq 0}$ dominate a non-trivial Bernoulli sequence?
\end{q}
\begin{rem}
A positive answer to Question \ref{q path} with a uniform bound on the  density $\rho>0$  would imply a positive answer to Question \ref{q walk}.
\end{rem}


Lastly, motivated by Proposition \ref{prop amenable} and Theorem \ref{thm spin-flip 2}, we state the following question.

\begin{q}\label{q domination}
Consider the upper stationary contact process $(\eta_t)$ on $\Latd$, $d\geq1$, with $\lambda>\lambda_c$. For which $\Delta \subset \Latd$ having ``zero density'' (that is, the l.h.s.\ of \eqref{eq density of set} equals $0$)   
does $(\eta_t)$ projected onto $\Delta\times [0,\infty)$ dominate a non-trivial independent spin-flip process?
\end{q}



\subsection*{Acknowledgment}
The authors thank Markus Heydenreich and Matthias Birkner for discussions and comments. S.A. Bethuelsen thanks LMU Munich for hospitality during the writing of the paper.
S.A. Bethuelsen was supported by the Netherlands Organization for Scientific Research (NWO).



\end{document}